\documentclass[10pt,a4paper,twoside]{article}
\usepackage{hyperref}
\usepackage{amsmath,amsfonts,amsthm,amsopn,color,amssymb}
\usepackage{palatino}
\usepackage{graphicx}

\newcommand{\eps}{\varepsilon}

\newcommand{\R}{\mathbb{R}}

\newcommand{\RN}{{\mathbb{R}^N}}
\newcommand{\RT}{{\mathbb{R}^3}}

\renewcommand{\le}{\leqslant}
\renewcommand{\ge}{\geqslant}

\renewcommand{\l }{\lambda}
\newcommand{\n }{\nabla }

\renewcommand{\t}{\theta}
\renewcommand{\O}{\Omega}

\renewcommand{\H}{H^1(\RN)}
\newcommand{\Hr}{H^1_r(\RN)}

\renewcommand{\P}{{\cal P}}

\newcommand{\D }{{\mathcal D}^{1,2}(\RN)}

\newcommand{\irn }{\int_{\RN}}
\newcommand{\irt }{\int_{\RT}}

\def\bbm[#1]{\mbox{\boldmath $#1$}}

\newtheorem{theorem}{Theorem}[section]
\newtheorem{lemma}[theorem]{Lemma}

\newtheorem{definition}[theorem]{Definition}

\newtheorem{remark}[theorem]{Remark}

\renewenvironment{proof}{\noindent{\textbf{Proof\quad}}}{$\hfill\square$\vspace{0.2 cm}\\}
\newenvironment{proofmain}{\noindent{\textbf{Proof of Theorem  \ref{main}\quad}}}{$\hfill\square$\vspace{0.2 cm}\\}
\newenvironment{proofth:app}{\noindent{\textbf{Proof of Theorem  \ref{th:app}\quad}}}{$\hfill\square$\vspace{0.2 cm}\\}

\title{{\bf The elliptic Kirchhoff equation in $\R^N$\\
perturbed by a local nonlinearity\footnote{The author is supported
by M.I.U.R. - P.R.I.N. ``Metodi variazionali e topologici nello
studio di fenomeni non lineari''}}}

\author{A. Azzollini \thanks{Dipartimento di Matematica ed Informatica, Universit\`a degli
Studi della Basilicata,  Via dell'Ateneo Lucano 10, I-85100 Potenza,
Italy, e-mail: {\tt antonio.azzollini@unibas.it}} }

\date{}

\begin{document}
    \maketitle

\begin{abstract}
    In this paper we present a very simple proof of the existence of at least one non trivial solution
    for a Kirchhoff type equation on $\RN$, for $N\ge 3$.
    In particular, in the first part of the paper we are interested
    in studying the existence of a positive solution to the elliptic Kirchhoff equation under the effect of
    a nonlinearity satisfying the general Berestycki-Lions
    assumptions.
    In the second part we look for ground
    states using minimizing arguments on a suitable natural constraint.
\end{abstract}
    {\it keywords:} {Kirchhoff equation, \and
    Pohozaev identity, \and natural constraint, \and minimizing sequence }

\section*{Introduction}

The multidimensional Kirchhoff equation is
    \begin{align}\label{eq:kirc}
        \frac{\partial^2 u}{(\partial t)^2} - (1+\int_\O |\n u|^2) \Delta u=0
    \end{align}
where $\O\subset\R^N$ and $u:\O\to\R$ satisfies some initial or
boundary conditions. It arises from the following Kirchhoff'
nonlinear generalization (see \cite{K}) of the well known d'Alembert
equation
    \begin{equation}\label{eq:kir}
        \rho\frac{\partial^2 u}{(\partial t)^2} - \left(\frac{P_0}{h}+\frac {E}{2L}\int_0^L \left|\frac{\partial u}{\partial x}\right|^2\,dx
        \right) \frac{\partial^2 u}{(\partial x)^2}=0.
    \end{equation}
Equation \eqref{eq:kir} describes a vibrating string, taking into
account the changes in length of the string during the vibration.
Here, $L$ is the length of the string, $h$ is the area of the cross
section, $E$ is the Young modulus of the material, $\rho$ is the
mass density and $P_0$ is the initial tension. In \cite{L} the
problem was proposed in the following form:
    \begin{equation*}
        \left\{
            \begin{array}{ll}
            \frac{\partial^2 u}{(\partial t)^2}-M(\int_\O|\n u|^2)\Delta u=f&\hbox{in }\O\times(0,T),
            \\
            u=0&\hbox{in }\partial\O\times(0,T)\\
            u(0)=u_0,\quad u'(0)=u_1,
            \end{array}
        \right.
    \end{equation*}
where $M:[0,+\infty[\to \R$ is a continuous function such that $M(s)\ge c>0$ for any $s\ge 0,$ and $\O$ is a bounded set in $\RN,$ with smooth boundary.\\
This hyperbolic problem has an elliptic version when we look for static solutions.\\
In \cite{V}, it has been considered a class of problems among which
the following elliptic Kirchhoff type equation was included
    \begin{equation*}
        \left\{
            \begin{array}{ll}
            -M(\int_\O|\n u|^2)\Delta u=f&\hbox{in }\O,
            \\
            u=0&\hbox{in }\partial\O
            \end{array}
        \right.
    \end{equation*}
(here $\O$ is an open subset of $\RN$).\\
Taking into account the original formulation of the equation given
by Kirchhoff, we assume the following
    \begin{definition}
        If there exist two positive constants $a$ and $b$ such that
        $M:\R_+\to\R$ can be written $M(s)=a+bs,$ then $M$ is called
        Kirchhoff function.
    \end{definition}
Recently, many authors have used variational methods to study the
Kirchhoff equation perturbed by a local nonlinear term (see \cite{M}
for a short survey on the topic). By arguments based on the mountain
pass theorem, in \cite{ACM} the problem
    \begin{equation*}
        \left\{
            \begin{array}{ll}
            -M(\int_\O|\n u|^2)\Delta u=f(x,u)&\hbox{in }\O,
            \\
            u=0&\hbox{in }\partial\O
            \end{array}
        \right.
    \end{equation*}
has been solved in a bounded domain of $\RN$ under suitable growth
conditions on $f:\RN\times\R\to\R$ and $M:[0,+\infty[\to\R.$ Taking
$N=1,2$ or $3$, the problem has been treated also in \cite{PZ} where
$M$ is a Kirchhoff function, and the nonlinearity $f(x,t)$ has been
supposed to behave linearly near 0 and like $t^3$ at infinity. The
Yang index has been used to find a nontrivial solution. A similar
growth at infinity has been assumed in \cite{PZ2} where the authors
have looked for sign changing solutions. They also have obtained a
sign changing solution when the nonlinearity $f$ satisfies either
the following growth condition
    $$|f(x,t)|\le C(|t|^{p-1}+1),\hbox{ uniformly in $x$, for } p<4$$ or
the following Ambrosetti-Rabinowitz condition
    \begin{equation*}
        \nu F(x,t)\le
        tf(x,t), \;\hbox{for } |t| \hbox{ large and }\nu >4.
    \end{equation*}
In \cite{MZ}, the equation has been studied assuming that the
nonlinearity grows at infinity more than $t^3$, without introducing
the Ambrosetti-Rabinowitz hypothesis. Using a variational approach,
a multiplicity result has been showed in \cite{HZ}. Finally we
recall the recent result obtained in \cite{R}, where three solutions
have been found for the Kirchhoff type problem
    \begin{equation*}
        \left\{
            \begin{array}{ll}
            -M(\int_\O|\n u|^2)\Delta u=\l f(x,u)+\mu g(x,t) &\hbox{in }\O,
            \\
            u=0&\hbox{in }\partial\O
            \end{array}
        \right.
    \end{equation*}
where $\l$ and $\mu$ are two parameters.

The joining point of all these papers is that they consider the equation on a bounded domain, with Dirichlet conditions on the boundary.\\
In this paper we study an autonomous Kirchhoff type equation on all
the space $\RN,$ looking for the existence of positive solutions,
namely we consider the problem
    \begin{equation}\tag{$\mathcal K$}\label{eq:main}
                \left\{
            \begin{array}{ll}
            -M(\int_\O|\n u|^2)\Delta u= g(u) &\hbox{in }\RN,\; N\ge
            3,
            \\
            u>0.
            \end{array}
        \right.
    \end{equation}
In this sense, the problem turns out to be a generalization of the
well known Schr{\"o}dinger equation
    \begin{equation}\tag{$\mathcal S$}\label{eq:schr}
        -\Delta u = g(u),\quad \hbox{in }\RN.
    \end{equation}
In the first part of the paper we are interested in studying the
problem \eqref{eq:main} in presence of a Berestycki-Lions
nonlinearity. In order to explain what this means, we provide the
following definitions
    \begin{definition}\label{def:BL}
        A function $g:\R\to\R$ is called a Berestycki-Lions nonlinearity if
        it satisfies the following assumptions
        \begin{itemize}
\item[({\bf g1})] $g\in C(\R,\R)$, $g(0)=0$;
\item[({\bf g2})]
$-\infty <\liminf_{s\to 0^+} g(s)/s\le \limsup_{s\to 0^+}
g(s)/s=-m<0$;
\item[({\bf g3})] $-\infty \le\limsup_{s\to +\infty}
g(s)/s^{2^*-1}\le 0$;
\item[({\bf g4})] there exists $\zeta>0$
such that $G(\zeta):=\int_0^\zeta g(s)\,d s>0$.
\end{itemize}

        A function $g:\R\to\R$ is called a zero mass Berestycki-Lions nonlinearity if
        it satisfies ({\bf g1}), ({\bf g3}), ({\bf g4}) and the following zero mass assumption

        ({\bf g2})' $-\infty <\liminf_{s\to 0^+} g(s)/s$,
$\quad\limsup_{s\to 0^+} g(s)/s^{2^*-1}\le 0$.
    \end{definition}

\begin{remark}
    Using the terminology inherited by \cite{BL1}, we
    refer to the constant $m$ in ({\bf g2}) calling it {\it mass}. This
    is the reason for which we say  that a function $g$ satisfying ({\bf g2})' instead of ({\bf g2})
    is a {\it zero mass} nonlinearity.
\end{remark}

In the very celebrated paper \cite{BL1}, these types of
nonlinearities appeared for the first time, and it was showed that
the hypotheses ({\bf g1}),$\ldots$,({\bf g4}) are {\it almost}
optimal to get an existence result for the Schr{\"o}dinger equation.

In the first section of this paper, we use a simple rescaling
argument to establish a sufficient condition for the existence of a
solution to \eqref{eq:main}. The first result we get is the
following
    \begin{theorem}\label{main}
        Let $M:\R_+\to\R_+\setminus\{0\}$ a continuous function such that
            \begin{equation}\label{eq:hyp}
                \liminf_{t\to 0} t M(t^{\frac{2-N}2})=0
            \end{equation}
        and let $g:\R\to\R$ a (possibly zero mass) Berestycki-Lions nonlinearity. Then the problem \eqref{eq:main} has a
        solution in $C^2(\RN)$.
    \end{theorem}

We point out that any Kirchhoff function satisfies our assumptions
if $N=3.$ On the other hand, when $M$ is a Kirchhoff type function
we are able to refine our estimates and we get the following result

    \begin{theorem}\label{th:app}
        Let $g:\R\to\R$ a (possibly zero mass) Berestycki-Lions nonlinearity, $f:\R_+\to\R_+$ a continuous
        function and $M(s)=a+bf(s).$
        Then for any $N\ge 3$ there exists a positive constant $\delta$ (depending on $a$) such that if
        $b\in ]0,\delta],$  the
        problem \eqref{eq:main} has a solution in $C^2(\RN).$ If
        moreover
        \begin{equation}\label{eq:hypp}
                \liminf_{t\to 0}  t^{\frac 2
                {2-N}}f(t)=0,
            \end{equation}
        then there exists a positive constant $\delta$ (depending on $b$) such that if
        $a\in ]0,\delta],$  the
        problem \eqref{eq:main} has a solution in $C^2(\RN).$
    \end{theorem}
We again remark that when $M$ is a Kirchhoff function, then
$f=id|_{\R_+}$ and assumption \eqref{eq:hypp} is automatically
satisfied for $N\ge 5$.

In the second part of the paper we study the existence of the so
called {\it ground state} solutions to the Kirchhoff equation
    \begin{equation}\label{eq:kirchhoff}
        -(a+b\irn|\n u|^2)\Delta u = g(u).
    \end{equation}
We recall that a ground state is a solution which minimizes the
functional of the action among all the other solutions.

The problem of finding such a type of solutions is a very classical
problem: it has been introduced by Coleman, Glazer and Martin in
\cite{CGM}, and reconsidered by Berestycki and Lions in \cite{BL1}
for a class of nonlinear equations including the Schr\"odinger's
one. Here we will use a minimizing argument based on an idea
developed in \cite{S} (recent applications can be found in \cite{PS}
and in \cite{RS}). In that paper the author showed that the ground
state for the Schr\"odinger equation can be found as the minimizer
of the functional of the action restricted to a particular natural
constraint. This natural constraint actually is a manifold which is
constituted by all non null functions satisfying the Pohozaev
identity related to the Schr\"odinger equation.\\
Unfortunately, a similar manifold turns out to be a nice constraint
in order to look for ground state solutions of the Kirchhoff
equation only if $N=3$ or $N=4:$ when $N\ge 5,$ we are no more able
to establish if the restricted functional is bounded below. The
final result we get is
    \begin{theorem}\label{th:grst}
        If $N=3$ or $N=4$ and $g$ is a Berestycki-Lions
        nonlinearity, then
        equation \eqref{eq:kirchhoff} possesses a ground state
        solution.
    \end{theorem}

The paper is so organized:

in Section \ref{sec:proof1} we show our rescaling argument to get a
solution for \eqref{eq:main} in the general case described in
Theorem \ref{main} and in the particular situation of a Kirchhoff
type function as in Theorem \ref{th:app}.

In Section \ref{sec:grst} we study the problem of the existence of a
ground state solution for the Kirchhoff equation using a variational
approach.

\section{Bound state solution}\label{sec:proof1}

In the sequel, we denote by $v$ a ground state solution of
\eqref{eq:schr} (respectively a bound state solution if $g$ is a
zero mass Berestycki-Lions nonlinearity).\\

\begin{proofmain}
    Observe that, by hypothesis \eqref{eq:hyp} and since
    $$\lim_{t\to +\infty} t M(t^{\frac{2-N}2})=+\infty,$$
    by continuity we have that there
    exists $\bar t>0$ such that $\bar t^2 M(\bar t^{2-N}\irn|\nabla v|^2)=1.$
    The function $u:\R^N\to\R$ defined as follows:
        $$x\in\RN\to v(\bar t x)\in\R,$$
    satisfies the equalities
        \begin{equation*}
        \left\{
            \begin{array}{l}
            {M({\irn|\nabla u|^2})} = \frac 1 { \bar t ^2}\\
            -\Delta u (x)= -\bar t^2 \Delta v (\bar t x) = \bar t^2 g(v(\bar t
        x))=\bar t^2 g(u(x)),
            \end{array}
            \right.
        \end{equation*}
    and then it is a solution of \eqref{eq:main}.
\end{proofmain}

\begin{remark}
    We point out that, since the only moment in which we use hypothesis \eqref{eq:hyp} is to determine the
    rescaling parameter $\bar t$, we can relax our assumption just requiring that
        \begin{equation}\label{eq:hyp2}
            \inf_{t\ge 0}t M\left(t^{\frac{2-N}2}\irn|\nabla
            v|^2\right)\le 1.
        \end{equation}
\end{remark}

\begin{proofth:app}
    As in the proof of Theorem \ref{main} we look for a solution to
    the equation
        \begin{equation*}
            t^2\left(a+bf(t^{2-N}\irn |\n v|^2)\right)=1.
        \end{equation*}
    Taking into account the previous remark, it is enough to prove
    that
    $$\inf_{t\ge 0}\Psi(t)\le 1,$$
    where $\Psi(t):=t\left(a+bf(t^{\frac{2-N}2}\irn |\n
    v|^2)\right).$
    Set  $$\bar h=f\left((2a)^{\frac{N-2}2}\irn |\n
    v|^2\right)$$
    and $\delta_1= \frac a {\bar h}. $ It is easy to verify that, if
    $b\le \delta_1,$ then
    $\Psi(1/2a)\le 1.$\\
    Now suppose that \eqref{eq:hypp} holds. We deduce that
    $$\liminf_{t\to +\infty} t f\left(t^{\frac{2-N}2}\irn|\n v|^2\right)=0.$$
    Let $\bar t$ such that $ \bar t f\left(\bar t^{\frac{2-N}2}\irn|\n v|^2\right)\le\frac 1
    {2b}$ and choose $a\le \delta_2=\frac 1 {2\bar t}.$ Again we have
    that $\Psi(\bar t)\le1.$
\end{proofth:app}

\section{Ground state solution}\label{sec:grst}

In this section we use a variational approach which requires some
preliminaries. In next subsection we will use the same arguments as
in \cite{BL1} to modify the nonlinearity $g$ in such a way we can
study equation \eqref{eq:main} looking for critical points of a
suitable functional.

\subsection{Functional framework}\label{sec:one}

Define $s_0:=\min\{s\in [\zeta,+\infty[\;\mid g(s)=0\}$
($s_0=+\infty$ if $g(s)\neq 0$ for any $s\ge\zeta$). We set $\tilde
g:\R\to\R$ the function such that
    \begin{equation}\label{eq:tilde}
      \tilde g(s)=\left\{
        \begin{array}{ll}
                g(s) &\hbox{ on } [0,s_0];
                \\
                0 &\hbox{ on } \R_+\setminus [0,s_0];
                \\
                (g(-s)-ms)^+ - g(-s) &\hbox{ on } \R_-.
      \end{array}
      \right.
    \end{equation}
By the strong maximum principle, if $u$ is a nontrivial solution of
\eqref{eq:main} with $\tilde g$ in the place of $g$, then $0< u<
s_0$ and so it is a positive solution of \eqref{eq:main}. Therefore
we can suppose that $g$ is defined as in \eqref{eq:tilde}, so that
({\bf g1}), ({\bf g2}), ({\bf g4}) and the following limit
    \begin{equation}\label{eq:limg}
        \lim_{s\to\pm\infty} \frac{g(s)}{s^{2^*-1}}=0
    \end{equation}
hold.
\\
We set
\begin{align*}
g_1(s) &:= \left\{
\begin{array}{ll}
(g(s)+ms)^+, & \hbox{if }s\ge0,
\\
0, & \hbox{if }s<0,
\end{array}
\right.
\\
g_2(s) &:=g_1(s)-g(s), \quad \hbox{for }s\in \R.
\end{align*}
Since
    \begin{align}
        \lim_{s\to 0}\frac{g_1(s)}{s} &= 0,\nonumber%\label{eq:lim1}
        \\
        \lim_{s\to\pm\infty}\frac{g_1(s)}{s^{2^*-1}}&=0,\label{eq:lim2}
    \end{align}
and
    \begin{equation}
        g_2(s) \ge ms,\quad  \forall s\ge 0,\label{eq:g2}
    \end{equation}
by some computations, we have that for any $\eps>0$ there exists
$C_\eps>0$ such that
    \begin{equation}
        g_1(s) \le C_\eps s^{2^*-1}+\eps g_2(s),\quad  \forall
        s\ge0\label{eq:g1g2}.
    \end{equation}
If we set
    \begin{equation*}
        G_i(t):=\int^t_0g_i(s)\,ds,\quad i=1,2,
    \end{equation*}
then, by \eqref{eq:g2} and \eqref{eq:g1g2}, we have
    \begin{equation}
         G_2(s) \ge \frac m 2 s^2,\quad  \forall s\in\R\label{eq:G2}
    \end{equation}
and for any $\eps>0$ there exists $C_\eps>0$ such that
    \begin{equation}
        G_1(s) \le \frac {C_\eps} {2^*} s^{2^*}+\eps G_2(s),\quad  \forall
        s\in\R\label{eq:G1G2}.
    \end{equation}

We define the functional
    \begin{equation*}
        I(u):=\frac  1 2 \tilde M(\|u\|^2)-\irt G(u)
    \end{equation*}
where we are denoting by $\|\cdot\|$ the norm
$\left(\irt|\n\cdot|^2\right)^{\frac 1 2 }$ of the space $\D,$ which
is the closure of the compactly supported smooth functions with
respect to the norm $\|\cdot\|.$ The previous functional is $C^1$ in
$\H,$ being $\H$ the closure of the compactly supported smooth
functions with respect to the norm
$$\|\cdot\|^2_{\H}=\irt|\n\cdot|^2+\irt|\cdot|^2.$$
We will look for critical points of the functional $I$ inside
$$\Hr:=\{u\in\H\mid u \hbox{ is radial}\},$$
which is a natural constraint for the functional $I$ by Palais'
principle of symmetric criticality. By standard variational
arguments, it is easy to prove that any critical point of $I$
corresponds to a weak solution of the equation. By the maximum
principle we will get a positive solution.

\subsection{Existence of a ground state solution}

We look for a ground state solution to
    \begin{equation*}
        -(a+b\|u\|^2)\Delta u =g(u), \quad u:\RN\to\R, \; N=3,\,4.
    \end{equation*}
A ground state of \eqref{eq:kirchhoff} is a nontrivial solution
$\bar u\in\H$ such that, if $v\in\H$ is another nontrivial solution
of \eqref{eq:kirchhoff}, then
    $$I(\bar u)\le I(v),$$
where $I:\H\to\R$ is the functional of the action related with
\eqref{eq:kirchhoff}, namely
    \begin{equation*}
        I(u)= \frac 1 2 \left( a +\frac b 2\irn|\n u|^2\right)\irn|\n u|^2-\irn G(u).
    \end{equation*}
Usually a standard technique to find a ground state consists in
looking for minimizers of the functional of the action restricted to
a natural constraint which contains all the possible solutions. A
candidate to play this role is the following Pohozaev set
    \begin{equation*}
        \P=\{u\in\H\setminus\{0\}\mid P(u)=0\}
    \end{equation*}
where for any $u\in\H$
    \begin{equation*}
        P(u)=a\frac{N-2}{2N}\irn |\n u|^2
        + b\frac{N-2}{2N}\left(\irn |\n u|^2\right)^2-\irn G(u).
    \end{equation*}

Actually the equality $P(u)=0$ is nothing but the Pohozaev identity
related with equation.

We will prove Theorem \ref{th:grst} following this scheme:
    \begin{itemize}
        \item[{\it step 1}:] we show that $\P$ is a $C^1$ manifold
        containing all the possible solutions of equation
        \eqref{eq:kirchhoff};
        \item[{\it step 2}:] we prove that $\P$ is a natural
        constraint, in the sense that every critical point of $I$
        restricted to $\P$ is a critical point of $I$;
        \item[{\it step 3:}] we show that $I|_\P$ is bounded below and
            $$\mu=\inf_{u\in\P}I(u)=\inf_{u\in\P}\frac 1 N\left(a\irn |\n u|^2+\frac{(4-N)b}{4}\left(\irn|\n
            u|^2\right)^2\right)$$
        is
        achieved.
    \end{itemize}

It is easy to see that $P$ is a $C^1$ functional. Moreover $\P$ is
nondegenerate in the
        following sense:
            \begin{equation*}
                \forall u\in\P:P'(u)\neq 0
            \end{equation*}
        so that $\P$ is a $C^1$ manifold of codimension one.
        Indeed, suppose by contradiction that $u\in\P$ and $P'(u)= 0,$
        namely $u$ is a solution of the equation
            \begin{equation}\label{eq:eqpo}
                -\left(a\frac{N-2} N + 2b\frac{N-2}N\|u\|^2\right)\Delta u =g(u).
            \end{equation}
        As a consequence, $u$ satisfies the Pohozaev identity referred to
        \eqref{eq:eqpo}, that is
            \begin{equation}\label{eq:pohopoho}
                a\frac{(N-2)^2}{2N^2}\irn |\n u|^2 + b \frac{(N-2)^2}{N^2}\left(\irn |\n u|^2\right)^2= \irn G(u).
            \end{equation}
        Since $P(u)=0,$ by \eqref{eq:pohopoho} we get
            \begin{equation*}
                -2a\irn |\n u|^2 + b (N-4)\left(\irn |\n
                u|^2\right)^2=0
            \end{equation*}
        and we conclude that $u=0$: absurd since $u\in\P.$ So $\P$
        is a $C^1$ manifold. It obviously contains all the solutions to
        \eqref{eq:kirchhoff} since every solution satisfies the
        Pohozaev identity $P(u)=0.$

        Now we pass to prove that $\P$ is a natural constraint for $I.$
        Suppose that $u\in\P$ is a critical point of the
        functional $I|_\P.$ Then, there exists $\l\in\R$ such
        that
            \begin{equation*}
                I'(u) = \l P'(u),
            \end{equation*}
        that is
            \begin{equation*}
                -(a + b \|u\|^2)\Delta u -g(u)=-\l (a\frac{N-2} N + 2b\frac{N-2}N\|u\|^2)\Delta u -\l
                g(u).
            \end{equation*}
        As a consequence, $u$ satisfies the following Pohozaev
        identity
            \begin{equation*}
                P(u) = \l a\frac{(N-2)^2}{2N^2}\irn |\n u|^2 +\l b \frac{(N-2)^2}{N^2}
                \left(\irn |\n u|^2\right)^2-\l \irn G(u)
            \end{equation*}
        which, since $P(u)=0$, can be written
            \begin{equation*}
               \l \left(-2a\irn |\n u|^2 + b(N-4)\left(\irn |\n u|^2\right)^2\right)= 0.
            \end{equation*}
        Since $u\neq 0,$ we deduce that $\l=0,$ and we conclude.

Now it remains to show that $\mu$ is achieved.

By the well known properties of the Schwarz symmetrization, we are
allowed to work on the functional space $\Hr$ as showed by the
following
    \begin{lemma}
        For any $u\in\P$ there exists $\tilde u\in\P\cap \Hr$ such
        that $I(\tilde u)\le I(u)$
    \end{lemma}
    \begin{proof}
        Let $u\in\P$ and set $u^*\in\Hr$ its symmetrized. It is
        easy to see that there exists $0<\tilde \t\le 1$ such that
        $\tilde u:=u^*(\cdot/\tilde\t)\in\P\cap\Hr$ and
        \begin{align*}
            I(\tilde u)&=\frac 1 N\left(a\irn |\n\tilde u|^2+\frac{(4-N)b}{4}\left(\irn|\n
            \tilde u|^2\right)^2\right)\\
            &=a
        \frac {\tilde \t^{N-2}} N \irn | \n u^*|^2+b\frac{(4-N)\tilde \t^{2(N-2)}}{4N}\left(\irn|\n
             u^*|^2\right)^2\\
             &\le
        \frac a N \irn | \n u^*|^2+\frac{(4-N)b}{4N}\left(\irn|\n
             u^*|^2\right)^2\\
             &\le I(u).
        \end{align*}
    \end{proof}
Before we proceed with the proof of the main result, another
preliminary result is required
    \begin{lemma}\label{le:b}
        $\mu:=\inf\{I(v)\mid v\in\P\}>0.$
    \end{lemma}
    \begin{proof}
        If $u\in\P$, then, by \eqref{eq:G1G2}, we have
            \begin{align*}
                C \| u\|^2&\le a\frac{N-2}2 \irn |\n u|^2+b\frac{N-2}{2}\left(\irn |\n u|^2\right)^2+ N
                (1-\eps)\irn G_2(u)\\
                &\le N C_\eps \irn | u|^{2^*}\le
                C' \| u\|^{2^*}
            \end{align*}
        where $\eps<1,$ $C_\eps,$ $C$ and $C'$ are suitable positive
        constants. We deduce that there exists a positive constant $C''$ such
        that $\|u \|\ge C''$ for any $u\in\P.$ The conclusion
        then follows once one observes that $I|_\P(u)\ge \tilde C\| u\|^2.$
    \end{proof}
\medskip
Now let $(u_n)_n$ be a minimizing sequence for $I|_\P$ in $\Hr,$
namely
    \begin{equation}\label{eq:hypotheses}
        \{u_n\}_n\subset\P\cap\Hr,\quad I(u_n)\to \mu.
    \end{equation}
Obviously $\| u_n\|$ is bounded. Moreover, since
$\{u_n\}_n\subset\P,$ certainly, by \eqref{eq:G1G2}, there exist
$0<\eps<1$ and $C_\eps >0$ such that
    \begin{equation*}
        a\frac{N-2}2 \irn |\n u|^2+b\frac{N-2}{2}\left(\irn |\n u|^2\right)^2+ N
                (1-\eps)\irn G_2(u)\le C_\eps N \|u_n\|_{2^*}^{2^*},
    \end{equation*}
and then we deduce also the boundedness of the $L^2-$norm of
$\{u_n\}_n$ by the continuous Sobolev embedding $\D\hookrightarrow
L^{2^*}(\RN)$ and \eqref{eq:G2}.

Let $u\in\Hr$ be the function such that, up to subsequences,
    \begin{equation}\label{eq:weak3}
        u_n\rightharpoonup u, \;\hbox{weakly in }\H.
    \end{equation}
We are going to prove that there exists $\bar\t>0$ such that
$$\bar u\in\P \hbox{ and } I(\bar u)=\mu$$
where $\bar u:= u(\cdot/\bar\t).$\\
Actually, by compactness due to the radial symmetry, from the weak
convergence \eqref{eq:weak3} we deduce
    \begin{equation}\label{eq:Strauss}
        \lim_n \irn G_1(u_n)=\irn G_1(u).
    \end{equation}
Of course, $u\neq 0.$ Otherwise, by \eqref{eq:Strauss} and since
$u_n\in\P$ for any $n\ge 1,$ we should have that
    \begin{equation*}
        0\le \limsup_n a\frac {N-2} 2 \|u_n\|^2 \le N \lim_n \irn G_1(u_n)=0,
    \end{equation*}
which, by \eqref{eq:hypotheses}, contradicts Lemma \ref{le:b}.\\
By \eqref{eq:Strauss}, the lower semicontinuity of the $\D-$norm and
the Fatou lemma, we have
    \begin{align*}
       a &\frac{N-2}{2}\irn |\n u|^2+b\frac{N-2}{2}\left(\irn |\n u|^2\right)^2+N\irn G_2(u)\\
       &\le\liminf_n \left(a \frac{N-2}{2}\irn |\n u_n|^2+b\frac{N-2}{2}\left(\irn |\n u_n|^2\right)^2+N\irn G_2(u_n)\right)\\
       &=\lim_n N\irn G_1(u_n)= N\irn G_1(u).
    \end{align*}
Let $0<\bar\t\le 1$ such that $\bar u= u(\cdot/\bar\t)\in\P.$ Using
the lower semicontinuity of the $\D-$norm, we infer that
    \begin{align*}
        I(\bar u) &= a\frac {\bar\t^{N-2}} N \irn | \n u|^2+b\frac{(4-N)\bar \t^{2(N-2)}}{4N}\left(\irn|\n
             u|^2\right)^2 \\
             &\le
        \liminf_n \frac {a} N \irn | \n u_n|^2+\frac{b(4-N)}{4N}\left(\irn|\n
             u_n|^2\right)^2 = \lim_n I(u_n) = \mu
    \end{align*}
and then we conclude.

\end{document}